\documentclass[12pt]{amsart}
\usepackage{times}

\usepackage{float}
\usepackage{geometry,graphicx,fancybox}
\usepackage{enumerate,amsmath,amssymb,amsthm}
\usepackage{comment}
\usepackage{epsfig}
\usepackage{pst-grad}
\usepackage{pst-plot}
\usepackage{mathtools}
\usepackage{enumitem}
\usepackage{bbold}
\usepackage{mathrsfs}

\theoremstyle{plain}

\usepackage[kerning=true]{microtype} 
\usepackage{graphicx}
\usepackage{wrapfig}
\usepackage{multirow}

\newtheorem{theorem}{Theorem}[section]
\newtheorem{corollary}[theorem]{Corollary}

\newtheorem{conjecture}[theorem]{Conjecture}
\newtheorem{proposition}[theorem]{Proposition}
\newtheorem{definition}[theorem]{Definition}

\theoremstyle{remark}

\theoremstyle{definition}
\newtheorem{remark}[theorem]{Remark}

\newcommand{\R}{\mathbb{R}}

\newcommand{\HH}{\mathbb{H}}
\newcommand{\Sm}{\mathscr{S}}

\newcommand{\omegatil}{\widetilde{\omega}}
\newcommand{\alphatil}{\tilde{\alpha}}

\newcommand{\Mtil}{\widetilde{M}}
\newcommand{\Btil}{\widetilde{B}}
\newcommand{\sigmatil}{\tilde{\sigma}}
\newcommand{\gtil}{\widetilde{g}}
\newcommand{\htil}{\tilde{h}}
\newcommand{\ytil}{\widetilde{y}}

\DeclareMathOperator{\vol}{vol}

\DeclareMathOperator{\sys}{sys}

\numberwithin{equation}{section}

\title[Macroscopic Schoen conjecture]{Macroscopic Schoen conjecture for manifolds with non-zero simplicial volume}

\author{F.~Balacheff and S.~Karam}

\address{F. Balacheff, Universitat Aut\`onoma de Barcelona, Spain.}

\email{fbalacheff@mat.uab.cat}

\address{S. Karam, Lebanese University, Lebanon.}

\email{karam.steve.work@gmail.com}

\thanks{The first author acknowledges support by grants ANR Finsler (ANR-12-BS01-0009-02) and Ram\'on y Cajal (RYC-2016-19334). The second author acknowledges support from grant ANR CEMPI (ANR-11-LABX-0007-01).}



\begin{document}
\maketitle

\begin{abstract} 
We prove that given a hyperbolic manifold endowed with an auxiliary Riemannian metric whose sectional curvature is negative and whose volume is sufficiently small in comparison to the hyperbolic one, we can always find for any radius at least $1$ a ball in its universal cover whose volume is bigger than the hyperbolic one. This result is deduced from a non-sharp macroscopic version of a conjecture by R.~Schoen about scalar curvature, whose proof is a variation of an argument due to M.~Gromov and based on a smoothing technique. 
We take the opportunity of this work to present a full account of this technique which involves simplicial volume and deserves to be better known. 
\end{abstract}

\bigskip

\section{Introduction}

L.~Guth proved in  \cite{Gu11} that there exists a positive constant $\delta_n$ such that, if $M^n$ is a closed hyperbolic manifold of dimension $n$ and $g$ is an auxiliary Riemannian metric with $\vol(M,g) < \delta_n \vol(M,hyp)$, then we can always find some ball of radius $1$ in its universal cover whose volume is bigger than $V_{hyp}(1)$. Here $M$ is said hyperbolic if it admits a metric of constant sectional curvature $-1$  and $V_{hyp}(R)$ denotes the hyperbolic volume of a geodesic ball of radius $R$ in $\HH^n$. Such a result gets some perspective once you translate the famous theorem of Besson, Courtois \& Gallot about volume entropy for hyperbolic manifolds  in the following way (see \cite{BCG95}): if $M$ is a closed hyperbolic manifold and $g$ is an auxiliary Riemannian metric with $\vol(M,g) < \vol(M,hyp)$, then there exists $R_0>0$ such that {\it any} ball of radius  $R\geq R_0$ in its Riemannian universal cover has volume bigger than $V_{hyp}(R)$. This leads Guth to conjecture the following.

\begin{conjecture}[Guth, 2011]
If $M$ is a closed hyperbolic manifold and $g$ is an auxiliary Riemannian metric such that $\vol(M,g) < \vol(M,hyp)$, then for all radius $R>0$
$$
\max_{y \in \Mtil} |\Btil(y,R)|> V_{hyp}(R).
$$
Here $|\Btil(y,R)|$ denotes the Riemannian volume of a metric ball $\Btil(y,R)\subset \Mtil$ in the universal cover with respect to the lifted metric $\gtil$. 
\end{conjecture}

While Besson, Courtois \& Gallot theorem  proves this conjecture for large radius (depending on the metric), the case where radius goes to zero is connected to a conjecture by R.~Schoen on scalar curvature. Indeed remember that for small radius $R>0$ the volume of a ball centered at some point $x$ in $M$ admits the following expansion:
$$
|B(x,R)|_g=\omega_nR^n\left(1-{Scal_g(x) \over 6(n+2)}R^2+o(R^3)\right).
$$
The condition $Scal_g(x)>Scal_{hyp}$ thus implies that $|B(x,R)|_g< V_{hyp}(R)$ provided $R$ is small enough. In particular Guth conjecture would imply via a scaling argument the conjecture attributed to Schoen\footnote{This conjecture would be a consequence of the fact that the Yamabe invariant of a closed hyperbolic manifold is achieved by its hyperbolic metric, which is conjectured to be true by Schoen in \cite[p.127]{Schoen89}.}  that a closed hyperbolic manifold with a Riemannian metric $g$ whose scalar curvature is at least equal to the scalar curvature of hyperbolic space has volume at least equal to the hyperbolic one. 

  The flat version of Guth conjecture is the generalized Geroch conjecture due to M. Gromov in 1985 (see \cite{Guth10}): for any $R>0$ you can always find in the universal cover of any Riemannian torus a ball of radius $R$ whose volume is at least equal to the euclidean one, that is $\omega_n R^n$.
The generalized Geroch conjecture, if it is true, would imply the original Geroch conjecture that tori do not admit Riemannian metrics with positive scalar curvature. The original Geroch conjecture is solved, see \cite{SY79} and \cite{GL80}.\\

Having this context in mind, the following result appears to be of some interest. 

\begin{theorem}\label{th:main}
There exists a constant $\alpha_n$ such the following estimate holds. If  $(M,g)$ is a closed Riemannian manifold of dimension $n$ with $\vol(M,g)< \alpha_n  \|M\|$, then for all $R\geq 1$
$$
\max_{y \in \Mtil} {|\Btil(y,R)| \over |\Btil(y,R/2)|}> {V_{hyp}(R) \over V_{hyp}(R/2)}.
$$
\end{theorem}

Here $\|M\|$ denotes the simplicial volume. We will recall this notion in the core of this paper, but for now the main property of this topological invariant that matters is that for hyperbolic manifolds the simplicial volume and the hyperbolic volume coincide up to some constant depending only on the dimension. This result due to Gromov and Thurston (see \cite{Thu78} and \cite{Gro82}) implies in particular that Riemannian manifolds which are also hyperbolic and whose volume is sufficiently small admits for any radius $R\geq 1$ some ball in their universal cover with some hyperbolic flavour.
Because scalar curvature can be alternatively defined as the following limit
$$
Scal_g(x)=8(n+2)\lim_{R \to 0}{1  \over R^2}\left(1-{|B(x,R)|\over 2^n |B(x,R/2)|}\right),
$$
we see that  an analog of Theorem \ref{th:main} valid for $R \to 0$ would imply a non-sharp version of Schoen conjecture. This is the reason why we interpret Theorem \ref{th:main} as a non-sharp {\it macroscopic} version of Schoen conjecture.\\

Theorem \ref{th:main} can also be seen as a stronger macroscopic version of the following theorem by Gromov \cite[Corollary, p. 36]{Gro82}: if the Ricci curvature is at least that of hyperbolic space, then the volume of the manifold is at least a constant times the simplicial volume. Indeed Bishop-Gromov inequality implies that with such a Ricci bound, the ratio between volumes of balls of radius $R$ and $R/2$ is at most the corresponding ratio in hyperbolic space.\\

Another argument in favor of Theorem \ref{th:main} is that it can be used to prove the following non-sharp version of Guth conjecture in the case of negative curvature for radii big enough.

\begin{theorem}\label{cor:main}
Let $M$ be a closed hyperbolic manifold of dimension $n$ and $g$ a Riemannian metric with negative sectional curvature.
There exists a positive constant $\beta_n$ such that the condition $\vol(M,g)< \beta_n \vol(M,hyp)$ implies that for all $R\geq 1$
$$
\max_{y \in \Mtil} |\Btil(y,R)|> V_{hyp}(R).
$$
\end{theorem}

This theorem was already known in dimension $n=2$ without any curvature bound condition, by a result of the second named author, see \cite{Kar15}.\\

The main ingredient in the proof of Theorem \ref{th:main} is a result by Gromov called the {\it smoothing inequality}, see \cite[pp. 33-34]{Gro82}.

This inequality relies on a dual definition of the simplicial volume of a manifold $M$ at the level of its universal covering. Loosely speaking, instead of minimizing the $\ell_1$-norm of simplicial cycles representing the fundamental class of $M$, Gromov explains in  \cite[pp. 28-31]{Gro82} how to define simplicial volume by minimizing the $\ell_\infty$-norm of some special cocycles in its universal covering. These special cocycles satisfy a property of {\it straightness} that mimics the caracteristics of a particular cocycle appearing in the hyperbolic context. The idea is that in the hyperbolic case, there exists one privileged representation of a simplex in its homotopy class with vertices fixed, namely the geodesic one. In the general Riemannian case, such a privileged representation is no longer well defined, but its dual version persists. 

The second ingredient of this smoothing inequality is, in the presence of a Riemannian metric, to replace points interpreted as Dirac measures by measures with density using the geometry inherited by the universal cover. More specifically, if $M$ is endowed with a Riemannian metric, Gromov essentially smooths Dirac measures of the universal covering into density measures supported by metric balls with an exponential decay. By combining this smoothing process with the dual definition of simplicial volume, Gromov obtains a smoothing inequality involving the Riemannian volume, the simplicial volume and the volume of balls in the universal cover.

 Gromov used this smoothing inequality to obtain mainly three results which compare for a hyperbolic manifold with an auxiliary Riemannian metric the following invariants: the Ricci curvature and the volume (see \cite[Corollary, p. 36]{Gro82}), the entropy and the volume (see \cite[p. 37]{Gro82}), and lastly the systole and the volume (see \cite[Theorem 6.4.D']{Gro83} and \cite[Theorem 3.B.1]{Gro96}). Our argument in the proof of Theorem \ref{th:main} appears as a variation of the one used by Gromov to prove the volume entropy estimate, see Remark \ref{rem:main}.
 See also subsection \ref{sec:concl} for a discussion about how Theorems \ref{th:main} and \ref{cor:main} interact with Gromov's results.\\

The first two sections of this paper are devoted to present this smoothing inequality. Some lectors will certainly prefer to read Gromov's text \cite{Gro82} which contains much more material. But we hope that our presentation may be of some help in complement of this lecture for the others. The alternative definition of simplicial volume using the notion of {\it straight invariant fundamental cocycle} is presented in section \ref{sec:cocycle}. Section \ref{sec:smoothing} presents the {\it smoothing} procedure of straight invariant fundamental cocycles, and a proof of the related smoothing inequality. Then we prove Theorem \ref{th:main} and Theorem \ref{cor:main} in section \ref{sec:proofs}.\\

\section{Simplicial volume via straight invariant cocycles}\label{sec:cocycle}

Let $M$ be a closed oriented $n$-dimensional manifold. Denote by $\widetilde{M}$ its universal cover and by 
$
\pi : \widetilde{M} \to M$
the corresponding covering map.\\

\subsection{Standard definition of simplicial volume.}
According to Mostow rigidity theorem \cite{Mos68} a hyperbolic metric on a manifold of dimension at least three is completely determined by the fundamental group of the manifold (up to isometry). 
The {\it simplicial volume} was originally designed by Gromov to find an explicit topological definition of volume for hyperbolic manifolds, and was first defined as the quantity
$$
\|M\|=\inf \{\sum_{i=1}^k |r_i| \mid \sum r_i \sigma_i \, \, \text{represents} \, \, [M]\},
$$
where the infimum is taken over all real singular chains representing the fundamental class of $M$. We will refer to the quantity $\sum_{i=1}^k |r_i|$ as the $\ell_1$-norm of the chain $\sum r_i \sigma_i$ in the sequel. 
If $M^n$ is a closed manifold of dimension $n\geq 2$ endowed with a hyperbolic metric $hyp$ then 
$$
\mathcal{V}_n \, \|M\|=\vol(M, hyp)
$$
where $\mathcal{V}_n$ denotes the maximal volume of an ideal $n$-simplex in the hyperbolic space $\HH^n$. This result due to both Gromov \cite{Gro82} and Thurston \cite{Thu78} ends the search for a topological definition of hyperbolic volume. It was also the main ingredient in Gromov's proof of Mostow rigidity theorem. 

For our purpose we only need to recall how to prove that $\mathcal{V}_n \, \|M\|\geq \vol(M, hyp)$. Because $M$ is hyperbolic it is provided with a straight operator. First define the notion of straight simplex of $\Mtil$ by induction as follows: the straight $k$-simplex of $\Mtil$ with vertices $(\ytil_0,\ldots,\ytil_k)$ is defined as the $\widetilde{hyp}$-geodesic cone over the straight simplex with vertices $(\ytil_0,\ldots,\ytil_{k-1})$. Then define the straight operator as the map which assigns to any singular $k$-simplex $\sigmatil$ of the universal cover $\widetilde{M}$ the unique straight simplex $\sigmatil_{st}$ with the same set of vertices. Using this straight operator, we see that any real singular chain $\sum r_i \sigma_i$ is homotopic to the straighted chain $\sum r_i \, \pi \circ (\tilde{\sigma}_i)_{st}$ where $\tilde{\sigma}_i$ denotes any lift of $\sigma_i$ to $\widetilde{M}$. If we represent the fundamental class of $[M]$ by such a chain, we get
$$
\vol(M,hyp)\leq \sum_{i=1}^k |r_i| \vol((\tilde{\sigma}_i)_{st},hyp)\leq  \sum_{i=1}^k |r_i| \mathcal{V}_n
$$
thus proving the desired inequality. Here observe that $\mathcal{V}_n$ can be alternatively defined as  the supremum of the volume of a straight $n$-simplex in the hyperbolic space $\HH^n$.\\

\subsection{Dual definition of simplicial volume.}
Next we need to recall the link between simplicial volume and bounded cohomology. For a cohomological class $\Omega \in H^n(M,\R)$ set $\|\Omega\|_\infty:=\inf_{[\omega]=\Omega}\sup_\sigma |\omega(\sigma)|$
where the supremum is taken over all simplices and the infimum over all real simplicial cocycles $\omega$ representing $\Omega$.
The following duality principle provides an alternative definition for simplicial volume and at the same time proves that manifolds with non-zero simplicial volume are exactly those whose dual fundamental class is bounded.

\begin{proposition}[Gromov]
$$
\|M\|=\sup \left\{\Omega[M]  \mid \Omega \in H^n(M,\R) \, \text{with} \,\|\Omega\|_\infty=1 \right\}.
$$
In particular, if $\Omega_M$ denotes the dual fundamental class of $M$, then
$$
\|M\|=(\|\Omega_M\|_\infty)^{-1}.
$$
\end{proposition}

\begin{proof}
For any real singular cycle $c=\sum_{i=1}^k r_i \sigma_i$ representing the class $[M]$, and any real singular cocycle $\omega$ representing a cohomological class $\Omega$ of dimension $n$, we have
$$
\Omega[M]=\omega(c)\leq \sum_{i=1}^k |r_i||\omega (\sigma_i)|\leq \sum_{i=1}^k |r_i|\cdot \|\omega\|_\infty,
$$
where $\|\omega\|_\infty:= \sup_\sigma |\omega(\sigma)|$ denotes the $\ell_\infty$-norm on the space of real cochains of $M$.
So  $\Omega[M]\leq \|M\|$ for any cohomological class $\Omega$ of unit $\ell_\infty$-norm.
If $\|M\|=0$ this concludes the proof.

If $\|M\|>0$, in the other direction, recall that the $\ell_1$-norm  on real chains is dual to the $\ell_\infty$-norm on real cochains. Let $c$ be a cycle representing $[M]$. According to the Hahn-Banach theorem there exists a linear form $\omega$ such that
\begin{itemize}
\item $\omega(c)=1$ ;
\item $\omega_{\mid \partial C_{n+1}(M,\R)} =0$ ;
\item $\|\omega\|_\infty= (d_{\|\cdot\|_1}(c,\partial C_{n+1}(M,\R)))^{-1}=\|M\|^{-1}$
\end{itemize}
where $C_{n+1}(M,\R)$ denotes the set of $(n+1)$-chains.
Therefore $\Omega:=[\omega]/ \|\omega\|_\infty$ is a cocycle with $ \|\Omega\|_\infty\leq 1$ such that $\Omega[M] =\|M\|$ which concludes the proof.
\end{proof}

If $g$ is a Riemannian metric on $M$, we could define the dual fundamental class using the following cocycle: for any simplex $\sigma$ set
$$
\omega(\sigma)={\vol(\sigma,g) \over \vol(M,g)}.
$$
Of course $[\omega]=\Omega_M$, but unfortunately $\|\omega\|_\infty=\infty$.
In the case of a hyperbolic manifold, we can fix this by using the straight operator. More precisely, if $(M,hyp)$ is hyperbolic, we define a slightly different cocycle $\omega_{hyp}$ using the formula
$$
\omega_{hyp}(\sigma)={\vol((\tilde{\sigma})_{st},\widetilde{hyp}) \over \vol(M,hyp)}
$$
 where $\tilde{\sigma}$ denotes any lift of $\sigma$ to $\widetilde{M}$.
It is straightforward to check that $[\omega_{hyp}]=\Omega_M$ and that 
$$
\|\omega_{hyp}\|_\infty={\mathcal{V}_n \over \vol(M,hyp)}={1 \over \|M\|}.
$$
The corresponding lifted cochain also enjoys the following properties.
\begin{proposition}
The $\pi_1M$-invariant cochain $\pi^\ast \omega_{hyp}$ is straight, that is for any simplex $\tilde{\sigma}$ 
$$
\pi^\ast \omega_{hyp}(\tilde{\sigma})=\pi^\ast \omega_{hyp}\left(\tilde{\sigma}_{st}\right).
$$
Equivalently,  the value $\pi^\ast \omega_{hyp}(\tilde{\sigma})$ only depends of $(y_0,\ldots,y_n)$, the set of vertices of $\tilde{\sigma}$.
Moreover, the resulting function
\begin{eqnarray*}
\pi^\ast \omega_{hyp} : \widetilde{M}^{n+1} & \to & \R\\
(y_0,\ldots,y_n) & \mapsto & \pi^\ast \omega_{hyp}(y_0,\ldots,y_n)
\end{eqnarray*}
is continuous and in particular Borel.
\end{proposition}

This is here that Gromov found that such nice properties of $\pi^\ast \omega_{hyp}$ could be used to define an alternative notion of simplicial volume. In the next paragraph we present this definition, but it is important to already underline that this alternative simplicial volume coincides with the standard one, albeit the proof of their equivalence seems quite technical.\\

\subsection{Alternative definition of simplicial volume via straight invariant cocycles.}
 \begin{definition}
 A straight invariant fundamental cocycle is a cochain $\tilde{\omega}$ of $C^n(\widetilde{M};\R)$ with the following properties:\\
 
\noindent a) \underline{Invariance}:  $\tilde{\omega}$ is $\pi_1M$-invariant,\\

\noindent b) \underline{Fundamental coc}y\underline{cle}:  the only cochain $\omega$ on $M$ satisfying $\pi^\ast \omega=\tilde{\omega}$ is a cocycle representing the dual fundamental class of $M$, that is $[\omega]=\Omega_M$. In particular $\tilde{\omega}$ is a cocycle.\\

\noindent c) \underline{Strai}g\underline{ht and Borel}:  $\tilde{\omega}$ is straight and the induced real valued function on $\Mtil^{n+1}$ is Borel.\\

The alternative simplicial volume is then defined by
$$
\|M\|'={1 \over \inf \|\tilde{\omega}\|_\infty}
$$
where the infimum is taken over all straight invariant fundamental cocycles.
 \end{definition}
 
In particular $\|M\|'=0$ if $M$ does not admit such a straight invariant fundamental cocycle of bounded type.
In the case of a hyperbolic manifold, we can compute the alternative simplicial volume as follows.

\begin{theorem}\label{th:GT}
If $M$ is hyperbolic, then
$$
\|M\|'={\vol(M,hyp) \over \mathcal{V}_n}.
$$
\end{theorem}

\begin{proof}
First observe that we already proved that 
$$
(\|M\|')^{-1}=\inf \|\tilde{\omega}\|_\infty \leq \|\pi^\ast \omega_{hyp}\|_\infty={\mathcal{V}_n \over \vol(M,hyp)}.
$$
In the reverse direction, consider a straight invariant fundamental cocycle $\tilde{\omega}$ of $\Mtil$ and denote by $\omega$ the corresponding cocycle of $M$ such that $\pi^\ast \omega=\tilde{\omega}$ and  $[\omega]=\Omega_M$. 
Following \cite[section 2.2]{Gro82} and \cite[Theorem 6.2]{Thu78} pick a straight simplex $\sigmatil_D$  of $\widetilde{M}=\HH^n$ all of whose edges have length equal to $D$ endowed with the natural orientation induced by $\Mtil$. Denote by $\tilde{\mu}$ the Haar measure on the group $\text{Isom}_+(\HH^n)$ of orientation-preserving isometries normalized so that the measure of isometries taking a point in $\HH^n$ to a region $R \subset \HH^n$ is the hyperbolic volume of $R$. This measure being invariant under both right and left multiplication, it descends to a measure denoted by $\mu$ on the quotient space $P(M)=\pi_1M \setminus \text{Isom}_+(\HH^n)$ such that $\int_{P(M)} d\mu=\vol(M,hyp)$. Denote by $F_D$ the set of simplices of $M$ obtained by projecting all straight simplices of $\HH^n$ all of whose edges have length $D$ with their natural orientation. The map
\begin{eqnarray*}
\Psi : P(M) & \to & F_D\\
\pi_1 M \varphi & \mapsto & \pi \circ \varphi \circ \sigmatil_D
\end{eqnarray*}
pushes forward the measure $\mu$ to a measure on $F_D$ denoted by $\text{smear}_M(\sigmatil_D)$. Now pick any reflection $r$ of $\HH^n$ associated to an hyperplane and set the measure on $F_D$ defined by
$$
\nu_D={1 \over 2} (\text{smear}_M(\sigmatil_D)-\text{smear}_M(r\circ \sigmatil_D)).
$$
By construction its total variation satisfies
$$
\|\nu_D\|=\|\text{smear}_M(\sigmatil_D)\|=\vol(M,hyp).
$$
The measure $\nu_D$ is a cycle of the measured homology of $M$ (see \cite[Section 6]{Thu78}): for any cocycle $\zeta \in C^n(M;\R)$ the integral 
$$
\int_{F_D} \zeta(\sigma) d \nu_D
$$
depends only on the cohomological class of $\zeta$ and we set by duality
$$
\langle [\nu_D],[\zeta]\rangle:=\int_{F_D} \zeta(\sigma) d \nu_D.
$$ 
The corresponding homological class $[\nu_D]$ is thus a multiple of the fundamental class. Because
$$
{1 \over \|\nu_D\|} \int_{F_D} \zeta_{hyp}(\sigma) d\nu_D = \vol(\sigmatil_D,\widetilde{hyp})
$$
where $\zeta_{hyp}$ denotes the cocycle induced by the hyperbolic volume form on $M$, we see that
$$
[\nu_D]= \vol(\sigmatil_D,\widetilde{hyp}) \cdot[M].
$$
From the fact that $[\omega]=\Omega_M$ we get that
$$
\vol(\sigmatil_D,\widetilde{hyp})=\int_{F_D} \omega(\sigma) d \nu_D \leq \|\tilde{\omega}\|_\infty \cdot \vol(M,hyp).
$$
By letting $D \to \infty$ we deduce that
$$
{\mathcal{V}_n \over \vol(M,hyp)} \leq \|\tilde{\omega}\|_\infty
$$
as $\lim_{D\to \infty} \vol(\sigmatil_D,\widetilde{hyp})=\mathcal{V}_n$. This proves the reverse inequality.
\end{proof}

In particular, the two simplicial volumes $\|\cdot\|'$ and $\|\cdot\|$ coincide for hyperbolic manifolds.
In \cite[section3]{Gro82}, Gromov proved that they are indeed always equal.
But his proof \textquotedblleft requires a bit of abstract machinery\textquotedblright, as he himself confessed (see \cite[p.30]{Gro82}). Because we don't need this equivalence between both definitions, we work from now with the alternative definition of simplicial volume that we abusively denote by $\|\cdot\|$.

\medskip

\begin{remark}
We assumed at the beginning of this section that $M$ was oriented. In fact simplicial volume does not depend on a specific choice of orientation on $M$. So the notion of simplicial volume of a closed {\it orientable} manifold is well defined. Furthermore, if $M$ is now supposed to be {\it non-orientable}, simply set
 $$
\|M\|={1 \over 2} \|M'\| 
 $$ 
 where $M'$ denotes the orientable double cover of $M$, and observe that Theorem \ref{th:GT} is still valid.
 \end{remark}
 
 \bigskip

\section{The smoothing inequality}\label{sec:smoothing}

Using the fact that straight invariant Borel cochains can be interpreted as functions with variables in $\Mtil$, Gromov defines a diffusion process associated to a family of probability measures living in the universal covering space of the manifold.\\

\subsection{Diffusion of straight invariant fundamental cocycles via a smoothing operator}
Denote by $\mathcal{M}$ the Banach space of finite measures $\mu$ on the universal cover $\Mtil$ of $M$, and by $\mathcal{P} \subset \mathcal{M}$ the subset of probability measures. We endow $\mathcal{M}$ with the usual norm
$$
\|\mu\|=\int_{\widetilde{M}}|\mu|
$$
where $|\mu|$ denotes the total variation.
Any straight invariant fundamental  cocycle $\tilde{\omega}$ uniquely extend to a $(n+1)$-linear function on $\mathcal{M}^{n+1}$ as follows 
$$
\tilde{\omega}(\mu_0,\ldots,\mu_n)=\int_{\Mtil^{n+1}} \tilde{\omega}(y'_0,\ldots,y'_n)d\mu_0(y'_0)\ldots d\mu_n(y'_n).
$$
Observe that
$$
\|\tilde{\omega}\|_{\infty}=\sup_{y_0,\ldots,y_n\in \Mtil}|\tilde{\omega}(y_0,\ldots,y_n)|=\sup_{\mu_i \in \mathcal{P}}|\tilde{\omega}(\mu_0,\ldots,\mu_n)|.
$$

\begin{definition}
A smoothing operator is a smooth $\pi_1 M$-equivariant function 
$$
\Sm : \Mtil \to \mathcal{P}.
$$
\end{definition}

The idea is to replace points of $\Mtil$ by probability measures, and to observe the effect of this diffusion on straight invariant fundamental cocycles.

\begin{theorem}
Given a smoothing operator $\Sm$ and a straight invariant fundamental cocycle $\tilde{\omega}$, the diffused cochain $\Sm^\ast \tilde{\omega}$ defined by
\begin{eqnarray*}
\Sm^\ast \tilde{\omega}(y_0,\ldots,y_n)& =&\tilde{\omega}\left({\Sm(y_0)},\ldots,{\Sm(y_n)}\right)
\end{eqnarray*}
is also a straight invariant fundamental cocycle.
\end{theorem}

\begin{proof}
It is clear that the straight (by definition !) cochain $\Sm^\ast \tilde{\omega}$ is $\pi_1M$-invariant and Borel.
To check that it is a fundamental cocycle, we first prove that $\Sm^\ast \omegatil$ is a cocycle. Indeed for any singular $(n+1)$-simplex $\sigmatil$ with vertices $(y_0,\ldots,y_{n+1})$ we have
\begin{eqnarray*}
\Sm^\ast  \omegatil(\partial \sigmatil)&=&\sum_{i=0}^{n+1} (-1)^i \Sm^\ast \omegatil(y_0,\ldots,\hat{y}_i,\ldots,y_{n+1})\\
&=& \sum_{i=0}^{n+1} (-1)^i \left(\int_{\Mtil^{n+1}} \omegatil(y'_0,\ldots,\hat{y}'_i,\ldots,y'_{n+1}) \, \,d\Sm(y_0) \ldots \widehat{d\Sm(y_i)} \ldots d\Sm(y_{n+1})\right)\\
&=& \sum_{i=0}^{n+1} (-1)^i \left(\int_{\Mtil^{n+2}} \omegatil(y'_0,\ldots,\hat{y}'_i,\ldots,y'_{n+1}) \, \, d\Sm(y_0) \ldots {d\Sm(y_i)} \ldots d\Sm(y_{n+1})\right)\\
&=&  \int_{\Mtil^{n+2}} \underbrace{\sum_{i=0}^{n+1} (-1)^i \omegatil(y'_0,\ldots,\hat{y'_i},\ldots,y'_{n+1})}_{=0} \, \, d\Sm(y_0) \ldots d\Sm(y_{n+1})\\
&=&0
\end{eqnarray*}
as $\omegatil$ is itself a cocycle. 

Now denote by $\Sm^\ast \omega$ the unique cocycle of $M$ such that $\pi^\ast (\Sm^\ast \omega)=\Sm^\ast \omegatil$. Fix a triangulation $\mathcal{T}$ of $M$ such that the cycle $\sum_{\sigma \in \mathcal{T}} \sigma$ defines the fundamental class of $M$. We can lift this cycle into a chain of $\Mtil$ as follows: for each $\sigma \in \mathcal{T}$ choose a lifted $n$-simplex denoted by $\sigmatil$ in such a way that the union of these lifted simplices gives a triangulation of the closure of some fundamental domain for the $\pi_1M$-action. 
Denote by $\{y_1,\ldots,y_k\}$ a subset of vertices of the chain $\sum_{\sigma \in \mathcal{T}} \sigmatil$ obtained by lifting all the vertices of $\mathcal{T}$. For all $\sigma \in \mathcal{T}$, there exists unique functions 
$$
f_\sigma : \{0,\ldots,n\} \to \{1,\ldots,k\}
$$
and 
$$
g_\sigma : \{0,\ldots,n\} \to \pi_1M 
$$
such that the vertices of $\sigmatil$ are described by the ordered set $\left(g_\sigma(0) \cdot y_{f_\sigma(0)},\ldots,g_\sigma(n) \cdot y_{f_\sigma(n)}\right)$. Observe that $\mathcal{T}$ being a triangulation $f_\sigma$ is one-to-one. Because
$$
\sum_{\sigma \in \mathcal{T}} \omega(\sigma)=\sum_{\sigma \in \mathcal{T}}  \omegatil\left(g_\sigma(0) \cdot y_{f_\sigma(0)},\ldots,g_\sigma(n) \cdot y_{f_\sigma(n)}\right)=1
$$
by assumption, we see that 
$$
\sum_{\sigma \in \mathcal{T}}  \omegatil\left(g_\sigma(0) \cdot y'_{f_\sigma(0)},\ldots,g_\sigma(n) \cdot y'_{f_\sigma(n)}\right)=1
$$
for any choice of $(y'_1,\ldots,y'_k) \in \Mtil^{k}$.
Then 
\begin{eqnarray*}
\sum_{\sigma \in \mathcal{T}}  \Sm^\ast \omega(\sigma)&=&\sum_{\sigma \in \mathcal{T}}  \Sm^\ast \omegatil(\sigmatil)\\
&=&\sum_{\sigma \in \mathcal{T}} \Sm^\ast \omegatil\left(g_\sigma(0) \cdot y_{f_\sigma(0)},\ldots,g_\sigma(n) \cdot y_{f_\sigma(n)}\right)\\
&=&\sum_{\sigma \in \mathcal{T}} \int_{\Mtil^{n+1}} \omegatil\left(y'_{f_\sigma(0)},\ldots,y'_{f_\sigma(n)}\right) \, \,d\Sm\left(g_\sigma(0) \cdot y_{f_\sigma(0)}\right)\ldots d\Sm\left(g_\sigma(n) \cdot y_{f_\sigma(n)}\right)\\
&=&\sum_{\sigma \in \mathcal{T}} \int_{\Mtil^{n+1}} \omegatil\left(g_\sigma(0) \cdot y'_{f_\sigma(0)},\ldots,g_\sigma(n) \cdot y'_{f_\sigma(n)}\right) \, \,d\Sm(y_{f_\sigma(0)})\ldots d\Sm(y_{f_\sigma(n)})\\
&=&\sum_{\sigma \in \mathcal{T}}  \int_{\Mtil^{k+1}}  \omegatil\left(g_\sigma(0) \cdot y'_{f_\sigma(0)},\ldots,g_\sigma(n) \cdot y'_{f_\sigma(n)}\right) \, \,d\Sm(y_1)\ldots d\Sm(y_k)\\
&=&\int_{\Mtil^{k+1}} \underbrace{\sum_{\sigma \in \mathcal{T}}  \omegatil\left(g_\sigma(0) \cdot y'_{f_\sigma(0)},\ldots,g_\sigma(n) \cdot y'_{f_\sigma(n)}\right)}_{=1}) \, \, d\Sm(y_1)\ldots d\Sm(y_k)=1.
\end{eqnarray*}

Thus $\Sm^\ast \omegatil$ is a fundamental cocycle.
\end{proof}

In particular we get the following.
\begin{corollary}
Given a smoothing operator $\Sm$,
$$
\|M\|={1 \over \inf \|\Sm^\ast \omegatil \|_\infty}
$$
where the infimum is taken over all straight invariant fundamental cocycles.
\end{corollary}

\begin{proof}
Simply observe that
$$
\|\Sm^\ast \tilde{\omega}\|_{\infty}\leq\|\tilde{\omega}\|_{\infty}
$$
for any straight invariant fundamental cocycle $\tilde{\omega}$.
\end{proof}

By defining $\Sm_{\delta}(y)=\delta_y$ for all $y$ in $\Mtil$ where $\delta_y$ denotes the Dirac function at $y$, we see that
$$
\Sm_{\delta}^\ast \omegatil=\omegatil
$$
for any straight cocycle, but $\Sm_{\delta}$ fails to be a smoothing operator as it is even not continuous. This example helps us to understand the meaning of {\it smoothing operator}: the cocycle $\omegatil$ is viewed as the singular object $\Sm_{\delta}^\ast \omegatil$ and we smooth it by replacing Dirac measures with probability measures depending smoothly on points.\\

\subsection{The smoothing inequality}

The main result we need is an inequality which links in presence of a Riemannian metric the Riemannian volume to the simplicial volume through smoothing operators. For this, given a Riemannian metric $g$ on $M$, we define the norm of the differential of $\Sm$ at some point $y \in \Mtil$ by
$$
\|d_y\Sm\|=\sup_\tau\|d_y\Sm(\tau)\|
$$
where $\tau$ runs over the unit tangent sphere $S_y\subset T_y \Mtil$ for the pulled-back metric $\gtil$.

\begin{theorem}[Gromov's smoothing inequality]\label{th:smoothing.inequality}
 Let $g$ be a  Riemannian metric on $M$, and $\Sm$ a smoothing operator such that 
 $$
\|d\Sm\|_\infty:=  \sup_{y\in \Mtil} \|d_y\Sm\|<\infty.
  $$
 Then
 $$
\|M\|\leq n! \|d \Sm\|_\infty^n \vol(M,g).
$$
\end{theorem}

\begin{proof}
If $\|M\|=0$ there is nothing to prove. So suppose that $\|M\|>0$ and fix a straight invariant fundamental cocycle $\tilde{\omega}$ with $\|\tilde{\omega}\|_\infty<\infty$. Observe that we can first antisymmetrize this cocycle by considering the following straight cochain
$$
\tilde{\omega}_{ant}(y_0,\ldots,y_n)={1 \over (n+1)!} \sum_\delta \, [\delta] \cdot \tilde{\omega}(y_{\delta(0)},\ldots, y_{\delta(n)})
$$ 
where $\delta$ runs over all permutations of $\{0,\ldots,n\}$ and where $[\delta]$ stands for the signature of $\delta$. The cochain $\tilde{\omega}_{ant}$ is still a straight invariant fundamental cocycle and satisfies
$$
\|\tilde{\omega}_{ant}\|_\infty\leq\|\tilde{\omega}\|_\infty.
$$
So let assume that $\tilde{\omega}$ is antisymmetric.

Now  we associate to $\omegatil$ a differential $n$-form on $\Mtil$ defined as follows.
 For $y \in \Mtil$ and $u_1,\ldots,u_n \in T_y \Mtil$, set
$$
\tilde{\alpha}_y(u_1,\ldots,u_n):=n ! \, \, \tilde{\omega}\left(\Sm(y),d_y\Sm(u_1),\ldots,d_y\Sm(u_n)\right).
$$
It is straightforward to check that the sup-norm
$$
\|\tilde{\alpha}\|_\infty:= \sup \tilde{\alpha}_y(u_1,\ldots,u_n)
$$
where the supremum is taken over all  $y \in \Mtil$ and $u_1,\ldots,u_n \in S_y$ is bounded as follows:
$$
\|\tilde{\alpha}\|_\infty\leq n! \, \, \|\tilde{\omega}\|_\infty \|d \Sm\|_\infty^n.
$$
Observe that $\tilde{\alpha}$ is $\pi_1M$-invariant and that the induced differential $n$-form $\alpha$ on $M$ lies in the cohomological class of $\Omega_M$. Indeed, first $\omegatil$ induces an $n$-form $\omegatil^\mathcal{M}$ on $\mathcal{M}$ by setting
$$
\omegatil^{\mathcal{M}}_\mu(\mu_1,\ldots,\mu_n)=n! \cdot \omegatil(\mu,\mu_1,\ldots,\mu_n)
$$
for any $\mu \in \mathcal{M}$ and $\mu_1,\ldots,\mu_n \in T_{\mu} \mathcal{M}\simeq \mathcal{M}$. Because $\omegatil$ is $(n+1)$-linear on $\mathcal{M}$, the equality
$$
\int_{[\mu_0,\ldots,\mu_n]} \omegatil^{\mathcal{M}}=\omegatil(\mu_0,\mu_1,\ldots,\mu_n)
$$
is satisfied for all linear $n$-simplex $[\mu_0,\ldots,\mu_n]$ of $\mathcal{M}$. 
Fix a triangulation  $\mathcal{T}$ of $M$. The cycle $\sum_{\sigma \in \mathcal{T}} \sigma$ which defines the fundamental class of $M$ lift to a chain of $\Mtil$ such that the union of the lifted simplices gives a triangulation of the closure of some fundamental domain for the $\pi_1M$-action. 
For all lifted simplex $\sigmatil$ denote by $\{y^\sigma_0,\ldots,y^\sigma_n\}$  its set of vertices.
From the equality
$$
\alphatil=\Sm^\ast (\omegatil^{\mathcal{M}})
$$
where $\Sm^\ast (\cdot)$ denotes the pullback operator on differential $n$-forms, we deduce that
$$
\int_M \alpha=\sum_{\sigma \in \mathcal{T}} \int_\sigma \alpha=\sum_{\sigma \in \mathcal{T}}\int_{\sigmatil} \alphatil=\sum_{\sigma \in \mathcal{T}} \int_{\sigmatil} \Sm^\ast (\omegatil^{\mathcal{M}})=\sum_{\sigma \in \mathcal{T}} \int_{\Sm\circ \sigmatil} \omegatil^{\mathcal{M}}.
$$
Because the $n$-simplex $\Sm\circ \sigmatil$ of $\mathcal{M}$ is homotopic to the linear simplex $[\Sm(y^\sigma_0),\ldots,\Sm(y^\sigma_n)]$, we get that
\begin{eqnarray*}
\int_M \alpha&=&\sum_{\sigma \in \mathcal{T}} \int_{[\Sm(y^\sigma_0),\ldots,\Sm(y^\sigma_n)]} \omegatil^{\mathcal{M}}\\
&=&\sum_{\sigma \in \mathcal{T}} \omegatil(\Sm(y^\sigma_0),\ldots,\Sm(y^\sigma_n))\\
&=&\sum_{\sigma \in \mathcal{T}} \Sm^\ast \omegatil(y^\sigma_0,\ldots,y^\sigma_n)=1
\end{eqnarray*}
and thus $[\alpha]=\Omega_M$.

By integration we get
$$
1= \int_M \alpha \leq \|\tilde{\alpha}\|_\infty \vol(M,g)\leq n! \, \, \|\tilde{\omega}\|_\infty \|d \Sm\|_\infty^n \vol(M,g).
$$
Taking the infimum over all straight invariant fundamental cocycles $\tilde{\omega}$ with $\|\tilde{\omega}\|_\infty<\infty$ leads to the theorem.
\end{proof}

\bigskip

\section{Application to volume estimates}\label{sec:proofs}

We now prove our main result Theorem \ref{th:main} and show how to deduce Corollary \ref{cor:main}. 

\subsection{Proof of Theorem \ref{th:main}}

Consider a closed Riemannian manifold $(M,g)$ with non-zero simplicial volume and suppose that
$$
\max_{y \in \Mtil} {|\Btil(y,R)| \over |\Btil(y,R/2)|}\leq {V_{hyp}(R) \over V_{hyp}(R/2)}
$$
for some $R\geq 1$.

Fix $\lambda>0$ and  define
$$
\tilde{\Sm}_{\lambda,R}(y)=(e^{-\lambda d_{\gtil}(y,y')}-e^{-\lambda R})\cdot \mathbb{1}_{\Btil(y,R)}(y')\cdot d\vol_{\gtil}(y').
$$
Here $\mathbb{1}_{\Btil(y,R)}$ denotes the charateristic function of a metric ball $\Btil(y,R)$ in $(\Mtil,\gtil)$ centered at $y$ and of radius $R$, while $d\vol_{\gtil}$ denotes the Riemannian volume density. It is straightforward to check that $\Sm_{\lambda,R}=\tilde{\Sm}_{\lambda,R}/\|\tilde{\Sm}_{\lambda,R}\|$ is a smoothing operator, and to observe that it diffuses points into a measure whose support lies on the metric ball $\Btil(y,R)$ with a exponential decreasing decay in terms of the distance to the center. 
The smoothing operator $\Sm_{\lambda,R}=\tilde{\Sm}_{\lambda,R}/\|\tilde{\Sm}_{\lambda,R}\|$  satisfies 
\begin{eqnarray*}
\|d_y\Sm_{\lambda,R}\| &= & {1 \over \|\tilde{\Sm}_{\lambda,R}(y)\|}\left\|d_y\tilde{\Sm}_{\lambda,R}- {d_y\|\tilde{\Sm}_{\lambda,R}(y)\| \over  \|\tilde{\Sm}_{\lambda,R}(y)\|}\tilde{\Sm}_{\lambda,R}\right\|\\
&\leq&{1   \over \|\tilde{\Sm}_{\lambda,R}(y)\|} \left(\|d_y\tilde{\Sm}_{\lambda,R}\|+\sup_\tau d_y\|\tilde{\Sm}_{\lambda,R}(y)\|(\tau)\right).
\end{eqnarray*}
Because 
$$
\|\tilde{\Sm}_{\lambda,R}(y)\|=\int_{\Btil(y,R)} (e^{-\lambda d_{\gtil}(y,y')}-e^{-\lambda R}) d\vol_{\gtil}(y'),
$$
we see that 
$$
d_y\|\tilde{\Sm}_{\lambda,R}(y)\|=-\lambda\cdot \int_{\Btil(y,R)} d_y d_{\gtil}(\cdot,y')e^{-\lambda d_{\gtil}(y,y')} d\vol_{\gtil}(y').
$$ 
In particular
$$
\sup_\tau d_y\|\tilde{\Sm}_{\lambda,R}(y)\|(\tau)\leq  \lambda\cdot \int_{\Btil(y,R)} e^{-\lambda d_{\gtil}(y,y')} d\vol_{\gtil}(y').
$$
But we also have that
$$
\|d_y\tilde{\Sm}_{\lambda,R}\|\leq  \lambda\cdot \int_{\Btil(y,R)} e^{-\lambda d_{\gtil}(y,y')} d\vol_{\gtil}(y')
$$
which leds to
$$
\|d_y\Sm_{\lambda,R}\| \leq 2\lambda \cdot {I(\lambda,R) \over I(\lambda,R)- |\Btil(y,R)|}
$$
with 
$$
I(\lambda,R):=\int_{\Btil(y,R)} e^{-\lambda(d(y,y')-R)} d\vol_{\gtil}(y').
$$
Our assumption on $R$ implies the following lower bound:
\begin{eqnarray*}
I(\lambda,R)&\geq&  e^{\lambda R\over 2} |\Btil(y,R/2)|\\
&\geq& e^{\lambda R\over 2} \cdot |\Btil(y,R)| \cdot {V_{hyp}(R/2) \over V_{hyp}(R)}.
\end{eqnarray*}
From this we derive that
$$
\|d_y\Sm_{\lambda,R}\| \leq2 \lambda \cdot {e^{\lambda R \over 2} \cdot  V_{hyp}(R/2) \over e^{\lambda R/2} \cdot  V_{hyp}(R/2) - V_{hyp}(R)},
$$
the function $I \mapsto I/(I-B)$ being strictly decreasing for a fixed positive constant $B$. Now by fixing
$$
\lambda = {2 \over R} \log\left({2 V_{hyp}(R) \over V_{hyp}(R/2)}\right) \Leftrightarrow {e^{\lambda R \over 2} \cdot  V_{hyp}(R/2) \over e^{\lambda R/2} \cdot  V_{hyp}(R/2) - V_{hyp}(R)}=2,
$$
we conclude that
$$
\|d_y\Sm_{\lambda,R}\|  \leq  {8 \over R} \log\left({2 V_{hyp}(R) \over V_{hyp}(R/2)}\right)=:f(R).
$$
The smoothing inequality gives that
$$
\|M\|\leq n! \left(f(R)\right)^n \vol(M,g).
$$
The asymptotic $V_{hyp}(R)\simeq  {Vol(S^{n-1}) \over 2^{n-1}} e^{(n-1)R}$ when $R \to \infty$ implies that 
$$
\lim_{R \to \infty}f(R)= 4(n-1).
$$
We thus define $C_n:=(\sup_{R\geq 1}f(R))^n<\infty$ and derive that the initial assumption that 
$$
\max_{y \in \Mtil} {|\Btil(y,R)| \over |\Btil(y,R/2)|}\leq {V_{hyp}(R) \over V_{hyp}(R/2)}
$$
for some $R\geq 1$ implies the inequality
$$
\|M\|\leq n! \, C_n \, \vol(M,g).
$$
This proves Theorem \ref{th:main} with $\alpha_n=1/\left(n! \,C_n\right)$.
\\

\begin{remark}
Remark that  
$$
f(R)\simeq {8(n+1)\log 2\over R}
$$
for $R \to 0$.
So this strategy fails for small $R$ and can not be used to prove a non-sharp version of Schoen conjecture.\\
\end{remark}

\begin{remark}\label{rem:main}
 In \cite[p.37]{Gro82}, Gromov used exactly this smoothing operator to compare the volume and the volume entropy by analyzing the asymptotic behaviour of the smoothing inequality when $R \to \infty$. 
\end{remark}

\subsection{Proof of Corollary \ref{cor:main}}
Let $M$ be a closed hyperbolic manifold and $g$ a Riemannian metric with negative sectional curvature.
Suppose that   $\vol(M,g)< \beta_n \vol(M,hyp)$ with
$$
\beta_n:={\alpha_n \over \lambda_n^n \cdot \mathcal{V}_n}
$$
where $\lambda_n\geq 2$ only depends on the dimension $n$ and will be determined in the sequel.\\

The conformal metric $h=\lambda_n^2 g$ on $M$ satisfies the condition
$$
\vol(M,h)=\lambda_n^n \vol(M,g)< \alpha_n \|M\|.
$$
 By Theorem \ref{th:main} for any $R\geq 1$ (recall that $\lambda_n\geq 2$ and so $\lambda_n R\geq 2$) we can find $y \in \Mtil$ such that
 $$
{|\Btil_{\htil}(y,\lambda_n R)|_{\htil} \over |\Btil_{\htil}(y, \lambda_n R/2)|_{\htil}}> {V_{hyp}(\lambda_n R) \over V_{hyp}(\lambda_n R/2)}.
 $$
 Here we have denoted by $|\cdot|_{\htil}$ the volume with respect to the pullback metric $\htil$ of $h$ on $\Mtil$, and by $\Btil_{\htil}(y,R)$ the metric ball of radius $R$ and center $y$ in $(\Mtil,\htil)$.  Because $|\Btil_{\htil}(y,\lambda_n R)|_{\htil}=\lambda_n^n |\Btil_{\gtil}(y,R)|_{\gtil}$ we deduce that
 $$
{|\Btil_{\gtil}(y,R)|_{\gtil} \over |\Btil_{\gtil}(y, R/2)|_{\gtil}}> {V_{hyp}(\lambda_n R) \over V_{hyp}(\lambda_n R/2)}.
 $$
 Now the function $R \mapsto {V_{hyp}(R) \over V_{hyp}(R/2)} e^{-(n-1)R/2}$ is positive and tends to $1$ when $R \to \infty$. Thus 
 $$
 c_n=\inf_{R\geq2}  {V_{hyp}(R) \over V_{hyp}(R/2)} e^{-(n-1)R/2}>0
 $$ 
 and we deduce that
 $$
 |\Btil_{\gtil}(y,R)|_{\gtil} > c_n e^{(n-1)\lambda_n R/2} \cdot |\Btil_{\gtil}(y, R/2)|_{\gtil}.
 $$
By \cite[Proposition 14]{Cr80} we know that for any positive $r$ less than half the injectivity radius of $(\Mtil,\gtil)$ 
$$
 |\Btil_{\gtil}(y,r)|_{\gtil}\geq c'_n r^n
$$
for some positive constant $c'_n$.
But negative sectional curvature and simply connectedness implies that the injectivity radius is infinite, and thus the inequality above holds for any $r$.
Thus
$$
|\Btil_{\gtil}(y,R/2)|_{\gtil}\geq  c'_n(1/2)^n
$$
and consequently
$$
|\Btil_{\gtil}(y,R)|_{\gtil}> c''_n e^{(n-1)\lambda_n R/2} 
$$
where $c''_n:=c_n\cdot c'_n \cdot (1/2)^n$.

Observe that 
$$
c''_n e^{(n-1)\lambda R/2}\geq V_{hyp}(R) \Leftrightarrow \lambda \geq {2\over (n-1)R} \log \left({V_{hyp}(R)\over c''_n}\right).
$$ 
Because the function $R \mapsto {2\over (n-1)R} \log \left({V_{hyp}(R)\over c''_n}\right)$ tends to $2$ when $R \to \infty$ we can define {\it a posteriori}
$$
\lambda_n:=\sup_{R\geq 1} {2\over (n-1)R} \log \left({V_{hyp}(R)\over c''_n}\right)\geq 2.
$$
That way we deduce that   for all $R\geq 1$
 $$
 |\Btil_{\gtil}(y,R)|_{\gtil} > V_{hyp}(R).
 $$ 
This proves Corollary \ref{cor:main}.\\

Remark that Corollary \ref{cor:main} holds if we change the assumption of hyperbolicity on $M$ by the non-vanishing condition of its simplicial volume (the hyperbolic volume in the volume upperbound bound condition being replaced by the simplicial volume). In dimension $3$, the Geometrization Conjecture implies that a closed manifold admits a negative sectional curved metric if and only if it is hyperbolic. In higher dimensions, there exist counterexamples due to Gromov \& Thurston, see \cite{GT87}. More generally, using the same approach for proving Corollary \ref{cor:main}, we deduce the following.

\begin{corollary}\label{cor:isoembolic}
Any closed Riemannian manifold  $M$  with $\|M\|>0$, injectivity radius at least $1$ and $\vol(M,g)< (\beta_n\cdot {\mathcal V}_n)\cdot \|M\|$ satisfies that for all $R\geq 1$
$$
V_g(R):=\sup_{y \in \Mtil} |\Btil(y,R)|> V_{hyp}(R).
$$
\end{corollary}

\bigskip

\subsection{Concluding remarks}\label{sec:concl}

Let us emphasize how Theorems \ref{th:main} and \ref{cor:main} interact with the results obtained by Gromov using this technique.\\

First, as already mentioned in the introduction, Gromov proved in \cite[Corollary, p. 36]{Gro82} the following theorem: if the Ricci curvature is at least that of hyperbolic space, then the volume of the manifold is at least a constant times the simplicial volume. Because Bishop-Gromov inequality implies that with such a Ricci bound, the ratio between volumes of balls of radius $R$ and $R/2$ is at most the corresponding ratio in hyperbolic space, Theorem \ref{th:main} appears as a stronger macroscopic version of this result.\\

Secondly, Gromov used the smoothing inequality in \cite{Gro82} to prove the following non-sharp version of Besson, Courtois \& Gallot theorem: {\it if $M$ is a closed hyperbolic manifold and $g$ a Riemannian metric satisfying $\vol(M,g) \leq (\alpha_n / {\mathcal{V}_n}) \cdot \vol(M,hyp)$, then
$$
h_{vol}(M,g)\geq h_{vol}(M,hyp).
$$
}
Here recall that $h_{vol}$ denotes the {\it volume entropy} (also known as {\it asymptotic volume}) defined by
$$
h_{vol}(M,g)=\lim_{R\to\infty} {\log |\Btil(y,R)|Ê\over R}
$$
where $y$ denotes any point in $\Mtil$. This volume entropy inequality is easily deduced from Theorem \ref{th:main}, as we use the same smoothing operator as in Gromov's proof.\\

Lastly, Gromov combined in \cite{Gro83} his smoothing technique with the existence of regular geometric cycles to prove a systolic inequality for manifolds with non-zero simplicial volume, see \cite[Theorem 6.4.D']{Gro83} and \cite[Theorem 3.B.1]{Gro96}. This systolic inequality is a central result in systolic geometry and can be described as follows: there exists a positive constant $C_n$ such that, if $(M,g)$ is a closed Riemannian manifold  $M$ with $\|M\|>0$, then 
$$
{\vol(M,g)  \over \sys(M,g)^n} \geq C_n \cdot {\|M\|  \over (\log \|M\|)^n}.
$$
Here $\sys$ denotes the systole, that is the least length of a non-contractible loop. Our Theorem \ref{cor:main} implies this result for negatively curved metrics. The proof of the existence of regular geometric cycles (sort of almost minimizing objects in systolic geometry) is quite delicate, see \cite{Bul15} for a complete and detailed proof. \\

\noindent {\bf Acknowledgements.} The authors are grateful to the referee for valuable comments, and to H.~Bray and L.~Guth for helpful discussions.


\end{document}